\documentclass[a4paper,12pt]{article}

\makeatletter
\@addtoreset{footnote}{page}
\makeatother

\usepackage{amsthm}
\usepackage{amsmath,amssymb,latexsym,amsfonts,mathrsfs}

\renewcommand{\hat}{\widehat}

\renewcommand{\bar}{\overline}

\theoremstyle{plain}
\newtheorem{Thm}{Theorem}[section]

\newtheorem{Lem}[Thm]{Lemma}

\newtheorem{Cor}[Thm]{Corollary}

\theoremstyle{definition}

\newtheorem{Rem}[Thm]{Remark}

\numberwithin{equation}{section}

\makeatletter
\renewcommand\section{\@startsection {section}{1}{\z@}%
                                   {-3.5ex \@plus -1ex \@minus -.2ex}%
                                   {2.3ex \@plus.2ex}%
                                   {\normalfont\large\bf}}
\makeatother

\makeatletter
\renewcommand\subsection{\@startsection {subsection}{1}{\z@}%
                                   {-3.5ex \@plus -1ex \@minus -.2ex}%
                                   {2.3ex \@plus.2ex}%
                                   {\normalfont\normalsize\bf}}
\makeatother

\begin{document}

\begin{center}
{\Large \bf 
Stable functional CLTs for scaled elephant random walks
}
\end{center}
\begin{center}
Go Tokumitsu
\end{center}

\begin{abstract}
We establish stable functional central limit theorems for scaled elephant random walks in the diffusive, critical, and superdiffusive cases using the martingale approach.
\end{abstract}


\section{Introduction}
The Elephant Random Walk (ERW), introduced by Sch\"utz--Trimper \cite{Sch} in 2004, is a one-dimensional discrete-time process with infinite memory. Throughout this paper, we consider the ERW defined on a probability space $(\Omega, \mathcal{F}, P)$. The model provides a simple example of anomalous diffusion and exhibits a phase transition:

Let $S_0=0$. The first step $X_1$ is defined by
\[X_1 = \begin{cases} 
+1 & (\text{with probability $q$}), \\
-1 & (\text{with probability $1-q$}).
\end{cases}
\]
For $n \ge 1$, let $U_n$ be a random variable uniformly distributed on $\{1, \dots, n\}$. The $(n+1)$-th increment $X_{n+1}$ is then defined as
\[
X_{n+1} = \begin{cases} 
+X_{U_n} & (\text{with probability $p$}), \\
-X_{U_n} & (\text{with probability $1-p$}),
\end{cases}
\]
where all the random choices are assumed to be independent. The position of the walk at time $n+1$ is given by $S_{n+1} = S_n + X_{n+1}$ for $n \ge 1$. The stochastic process $(S_n)_{n=0}^\infty$ defined in this manner is called the Elephant Random Walk (ERW). The ERW has been extended to various spatial structures, including multidimensional lattices \cite{MERW1, MERWFCLT} and more general periodic structures \cite{shibata2025functional}. However, in this paper, we restrict our attention to the one-dimensional case.

Sch\"utz and Trimper \cite{Sch} showed that the ERW exhibits a phase transition and anomalous diffusion:
\begin{align*}
E[S_n^2]\sim\begin{cases}
(3-4p)^{-1}n&(p<3/4)\\
n\log n&(p=3/4)\\
((4p-3)\Gamma(4p-2))^{-1}n^{4p-2}&(p>3/4)
\end{cases},
\end{align*}
where $x_n\sim y_n$ means that $x_n/y_n\to1$ as $n\to\infty$. Based on this asymptotic behavior of the second moment, the ERW is called diffusive when $p<3/4$ and superdiffusive when $p>3/4$.

Let $\mathcal{F}_n = \sigma(X_1, \dots, X_n)$ be the natural filtration. Although the ERW may look like a non-Markovian process because it remembers its entire past, it is actually a time-inhomogeneous Markov chain. Indeed,  the conditional distribution of $X_{n+1}$ is determined by its conditional expectation:
\begin{align}
E[X_{n+1} | \mathcal{F}_n] = \frac{(2p-1)S_n}{n} \quad \text{a.s.}
\end{align}
Since $X_{n+1} \in \{-1, +1\}$, the transition probabilities are given by
\begin{align}\label{Markov}
P(S_{n+1}=S_n \pm 1 | S_0, S_1, \dots, S_n) = \frac{1}{2} \pm \left(p - \frac{1}{2}\right) \frac{S_n}{n}.
\end{align}

To analyze the fluctuations of the ERW, we introduce a sequence $(a_n)_{n=1}^\infty$ defined by
\begin{align*}
a_n := \prod_{k=1}^{n-1} \left( 1 + \frac{2p-1}{k} \right)^{-1} = \frac{\Gamma(2p)\Gamma(n)}{\Gamma(n+2p-1)}.
\end{align*}
It is well known that $\{M_n, \mathcal{F}_n\}_{n=0}^\infty$ defined by $M_0 = 0$ and $M_n = a_n S_n$ for $n \ge 1$ is a martingale. By Stirling's formula for the gamma function, the $(a_n)$ exhibit the following asymptotic behavior:
\begin{align*}
a_n \sim \frac{\Gamma(2p)}{n^{2p-1}} \quad \text{as } n \to \infty.
\end{align*}

Our main goal in this paper is to establish functional central limit theorems (FCLTs) with stable convergence for scaled ERW by martingale approach.

\subsection{Main results}
Throughout this paper, the symbol ``$\Longrightarrow$'' denotes convergence in distribution. Let $D:=D([0,\infty))$ denote the Skorokhod space, i.e., the set of right-continuous functions with left-hand limits equipped with the Skorokhod topology, and let $C\subset D$ denote the subset of continuous functions.  All joint convergences stated below imply stable convergence, and in particular mixing convergence, with respect to $\mathcal{F}$; see \cite{mixings, mixingrv, stableev} for the original contributions on stable and mixing convergence. For definitions and further details, we refer the reader to \cite{stableconv, JacodShiryaev}. 

We state our main theorems about scaled ERWs in several cases. The ways of scaling and the limit processes may vary according to each case.

First, we treat with the superdiffusive case. We focus on the scaled ERW $((1+t)^{1-2p}S_{n+[nt]}-S_n)/\sqrt{n}$, where we subtract $S_n$ from the shifted process $(1+t)^{1-2p}S_{n+[nt]}$, as follows.
\begin{Thm}\label{mThm}
Let $(S_n)$ be the ERW with $3/4<p<1$. Then the following joint convergence in $D\times \mathcal{Y}$ holds:
\begin{align}\label{mFCLT}
\left(\left(\frac{(1+t)^{1-2p}S_{n+[nt]}-S_n}{\sqrt{n}}\right)_{t\ge0},Y\right)\Longrightarrow \left(\left(\frac{1}{\sqrt{4p-3}}W(1-(1+t)^{3-4p})\right)_{t\ge0},Y\right),
\end{align}
where $Y$ is an arbitrary $\mathcal{F}$-measurable random variable taking values in a separable metrizable space $\mathcal{Y}$ and where $W=(W(t))_{t\ge0}$ is a standard Brownian motion independent of $Y$. In particular, we may take $Y=L_{p,q}$, where $L_{p,q}$ denotes the almost sure limit of $S_n/n^{2p-1}$ given in \eqref{SLLN}.
\end{Thm}
\begin{Rem}
\begin{enumerate}
\item By the previous result by Coletti--Gava--Sch\"utz (Theorem \ref{nondegThm}),
\begin{align*}
\frac{(1+t)^{1-2p}S_{n+[nt]}-S_n}{n^{2p-1}}
\xrightarrow[n\to\infty]{a.s.}
L_{p,q}-L_{p,q}=0,
\end{align*}
so Theorem \ref{mThm} describes the $\sqrt{n}$ order fluctuations of $(1+t)^{1-2p}S_{n+[nt]}$ around $S_n$. In contrast, the previous result of the CLT by Kubota--Takei \cite{KaT} (Theorem \ref{KubotaTakei}) concerns the fluctuations of $S_n$ around the random drift $n^{2p-1}L_{p,q}$.
\item This result reveals that the convergence in the first component of \eqref{mFCLT} is not affected by the random variable $L_{p,q}$.
\end{enumerate}
\end{Rem}
Second, we treat with the diffusive case. The following theorem states the FCLT for diffusive ERWs in the $\sqrt{n}$ order, which extends the previous result of the FCLT by Baur--Bertoin (Theorem \ref{FCLTERW1}) to stable convergence.
\begin{Thm}\label{FCLT3stable}
Let $(S_n)$ be the ERW with $0\le p<3/4$. Then the following joint convergence in $D\times \mathcal{Y}$ holds:
\begin{align}\label{sFCLT3}
\left(\left(\frac{S_{[nt]}}{\sqrt{n}}\right)_{t\ge0},Y\right)\Longrightarrow \left(\left(\frac{1}{\sqrt{3-4p}}t^{2p-1}W(t^{3-4p})\right)_{t\ge0},Y\right),
\end{align}
where $Y$ and $W$ are those in Theorem \ref{mThm}.
\end{Thm}

The following corollary states the FCLT for scaled diffusive ERWs, which follows by applying the Continuous Mapping Theorem to the convergence in distribution \eqref{sFCLT3}.
\begin{Cor}\label{FCLTjoint1} Let $(S_n)$ be the ERW with $0\le p<3/4$. Then the following joint convergence in $D\times D\times\mathcal{Y}$ holds: 
\begin{align}
&\left(
\left(\frac{S_{[nt]}}{\sqrt{n}}\right)_{t\ge0},
\left(\frac{(1+t)^{1-2p}S_{n+[nt]}-S_n}{\sqrt{n}}\right)_{t\ge0}
,Y\right)
\nonumber\\
&\qquad\Longrightarrow
\left(
\left(\frac{1}{\sqrt{3-4p}}t^{2p-1}W(t^{3-4p})\right)_{t\ge0},
\left(\frac{1}{\sqrt{3-4p}}W((1+t)^{3-4p}-1)\right)_{t\ge0}
,Y\right),
\end{align}
where $Y$ and $W$ are those in Theorem \ref{mThm}.

\end{Cor}

\begin{Rem}
The ERW with $p=1/2$ reduces to the simple symmetric random walk by Eq. \eqref{Markov}. In this case, the FCLT for the scaled ERW follows immediately from Donsker's invariance principle together with the stationarity of increments. Indeed,
\begin{align*}
(S_{n+[nt]}-S_n)_{t\ge0}\stackrel{d}{=}(S_{[nt]})_{t\ge0}.
\end{align*}
In contrast, the ERW with $p\neq1/2$ is time-inhomogeneous by Eq. \eqref{Markov}. This property is explicitly reflected in the FCLT for the scaled ERW, where the scaling factor $(1+t)^{1-2p}$ is required as a correction.
\end{Rem}

Third, we treat with the critical case.
The following theorem states the FCLT for the process $(S_n/\sqrt{n})$ in the critical case with a scaling of order $\sqrt{\log n}$, which extends the previous result of the FCLT by Baur--Bertoin (Theorem \ref{FCLTERW2}) to stable convergence.
\begin{Thm}\label{FCLT4stable}
Let $(S_n)$ be the ERW with $p=3/4$. Then the following joint convergence in $D\times \mathcal{Y}$ holds:
\begin{align}\label{sFCLT4}
\left(\left(\frac{S_{[n^t]}}{\sqrt{n^t\log n}}\right)_{t\ge0},Y\right)\Longrightarrow((W(t))_{t\ge0},Y).
\end{align}
where $Y$ and $W$ are those in Theorem \ref{mThm}.
\end{Thm}

The following corollary states the FCLT for scaled critical ERW, which follows by applying the Continuous Mapping Theorem to the convergence in distribution \eqref{sFCLT4}.
\begin{Cor}\label{FCLTjoint2} Let $(S_n)$ be the ERW with $p=3/4$. Then the following joint convergence in $D\times D\times\mathcal{Y}$ holds: 
\begin{align}
&\left(
\left(\frac{S_{[n^t]}}{\sqrt{n^t\log n}}\right)_{t\ge0},
\left(\frac{1}{\sqrt{n\log n}}\left(\sqrt{\frac{n}{n+[n^t]}}S_{n+[n^t]}-S_n\right)\right)_{t\ge0}
,Y\right)
\nonumber\\
&\qquad\Longrightarrow
\left(
(W(t))_{t\ge0},
(W((1\vee t)-1))_{t\ge0}
,Y\right),
\end{align}
where $Y$ and $W$ are those in Theorem \ref{mThm}.

\end{Cor}
\subsection{Previous results}
We first summarize the known limit theorems for the ERW. The CLTs for the diffusive and critical cases were established by Coletti et al. \cite{Col2}.

\begin{Thm}[Coletti--Gava--Sch\"utz {\cite[Theorem 2]{Col2}}]
\begin{enumerate}
\item When $0\le p<3/4$, the following CLT holds:
\begin{align*}
\frac{S_n}{\sqrt{n}} \Longrightarrow N\left(0, \frac{1}{3-4p}\right).
\end{align*}
\item When $p=3/4$, the following CLT holds:
\begin{align*}
\frac{S_n}{\sqrt{n\log n}} \Longrightarrow N(0,1).
\end{align*}
\end{enumerate}
\end{Thm}

In the superdiffusive case, where the CLT breaks down, the following scaling limit is obtained.
\begin{Thm}[Coletti--Gava--Sch\"utz {\cite[Theorem 3]{Col2}}]\label{nondegThm}
Let $(S_n)$ be the ERW with $3/4<p\le 1$. Then we have
\begin{align}\label{SLLN}
\frac{S_n}{n^{2p-1}}\to L_{p,q}\quad\text{a.s. and in $L^2$}
\end{align}
where $L_{p,q}$ is a non-degenerate random variable.
\end{Thm}
Regarding the fluctuations around the $n^{2p-1}L_{p,q}$ in the superdiffusive case, Kubota--Takei \cite{KaT} established the following CLT.
\begin{Thm}[Kubota--Takei {\cite[Theorem 3]{KaT}}]\label{KubotaTakei}
Let $(S_n)$ be the ERW with $3/4<p<1$. Then we have 
\begin{align}
\frac{S_n-n^{2p-1}L_{p,q}}{\sqrt{n}}\Longrightarrow N\left(0,\frac{1}{4p-3}\right).
\end{align}
\end{Thm}
The following theorem states the FCLT for the ERW in the $\sqrt{n}$ order.
\begin{Thm}[Baur--Bertoin {\cite[Theorem 1]{BaurBertoin}}]\label{FCLTERW1}
Let $(S_n)$ be the ERW with $0\le p<3/4$. Then we have the following distributional convergence in $D:$
\begin{align}\notag
\left(\frac{S_{[nt]}}{\sqrt{n}}\right)_{t\ge0}\Longrightarrow\left(\frac{1}{\sqrt{3-4p}}t^{2p-1}W(t^{3-4p})\right)_{t\ge0}.
\end{align}
\end{Thm}
The following theorem states the FCLT for the process $(S_n/\sqrt{n})$ in the critical case with a scaling of order $\sqrt{\log n}$.
\begin{Thm}[Baur--Bertoin {\cite[Theorem 2]{BaurBertoin}}]\label{FCLTERW2}
Let $(S_n)$ be the ERW with $p=3/4$. Then we have the following distributional convergence in $D([0,\infty)):$
\begin{align}\notag
\left(\frac{S_{[n^t]}}{\sqrt{n^t\log n}}\right)_{t\ge0}\Longrightarrow(W(t))_{t\ge0}.
\end{align}
\end{Thm}
\begin{Rem}
The proofs of Theorems \ref{FCLTERW1} and \ref{FCLTERW2} in \cite{BaurBertoin} 
rely on the general framework of generalized P\'olya urns developed in \cite{Janson}. 
The same framework has also been used to establish FCLTs 
for MERWs \cite{MERWFCLT} and for ERWs on periodic structures \cite{shibata2025functional}. 
These results establish convergence in distribution.

\end{Rem}
\subsection{Organization of this paper}
This paper is organized as follows. In Section 2, we provide the proofs of our main results. In the Appendix, we present a stable martingale FCLT and provide proofs of the technical lemmas.
\section{Proofs of main results}
The martingale difference sequence $\Delta M_n := M_n - M_{n-1}$ is then given by
\begin{align}
\Delta M_n=a_n(X_n-E[X_n|\mathcal{F}_{n-1}]).
\end{align}
Since $|X_n|=1$, we have
\begin{align}
|\Delta M_n|\le2a_n.
\end{align}
\subsection{Superdiffusive case}
We now prove Theorem \ref{mThm}.
\begin{proof}[Proof of Theorem \ref{mThm}]Set
\begin{align*}
s_n:=\frac{\Gamma(2p)}{\sqrt{4p-3}}n^{3/2-2p}
\end{align*}
and 
\begin{align*}
X_{n,k}:=s_n^{-1}\Delta M_{n+k}\quad(n, k\ge1),\quad \mathcal{F}_{n,k}:=\mathcal{F}_{n+k} \quad(n\ge1, k\ge0).
\end{align*}
Then, $\{X_{n,k},\mathcal{F}_{n,k}\}$ is a martingale difference array, and $\mathcal{F}_{n,k}$ increases as $n$ increases. Now, Let us check Conditions (a) and (b) in Theorem \ref{martFCLT}.

(Verification of Condition (a)): We rewrite the sum of conditional variances of the increments as follows:
\begin{align}
\notag\sum_{k=1}^{[nt]}E[X_{n,k}^2|\mathcal{F}_{n,k-1}]&=\frac{1}{s_n^2}\sum_{k=n+1}^{n+[nt]}E[(\Delta M_k)^2|\mathcal{F}_{k-1}]\\
\label{sum1}&=\frac{1}{s_n^2}\sum_{k=n+1}^{n+[nt]}a_k^2(1-E[X_k|\mathcal{F}_{k-1}]^2).
\end{align}
By Theorem \ref{nondegThm}, we have
\begin{align*}
E[X_{n+1}|\mathcal{F}_{n}]=\frac{(2p-1)S_{n}}{n}=\frac{(2p-1)S_n}{n^{2p-1}}\cdot n^{2p-2}\xrightarrow[n\to\infty]{a.s.}0.
\end{align*}
Hence, the asymptotic behavior of the sum of \eqref{sum1} is determined by the sum of $a_k^2$ as follows:
\begin{align*}
\sum_{k=1}^{[nt]}E[X_{n,k}^2|\mathcal{F}_{n,k-1}]&\sim \frac{1}{s_n^2}\sum_{k=n+1}^{n+[nt]}a_k^2\sim\frac{4p-3}{\Gamma(2p)^2}n^{4p-3}\sum_{k=n+1}^{n+[nt]}\Gamma(2p)^2 k^{2-4p}\\
&\sim n^{4p-3}(n^{3-4p}-(n+[nt])^{3-4p})\\
&\xrightarrow[n\to\infty]{} 1-(1+t)^{3-4p}.
\end{align*}

(Verification of Condition (b)): The following inequality holds:
\begin{align*}
\sum_{k=1}^{[nt]} E[X_{n,k}^21_{\{|X_{n,k}|>\varepsilon\}}|\mathcal{F}_{n,k-1}]&\le\frac{1}{\varepsilon^2}\sum_{k=1}^{[nt]}E[X_{n,k}^4|\mathcal{F}_{n,k-1}]\\
&=\frac{1}{\varepsilon^2 s_n^4}\sum_{k=n+1}^{n+[nt]}E[(\Delta M_k)^4|\mathcal{F}_{k-1}]\\
&\le\frac{16}{\varepsilon^2 s_n^4}\sum_{k=n+1}^{\infty}a_k^4\sim\frac{16(4p-3)^2}{\varepsilon^2(8p-5)}n^{-1}\xrightarrow[n\to\infty]{}0.
\end{align*}
Therefore, we obtain the following joint convergence in $D\times \mathcal{Y}$:
\begin{align*}
\left(\left(\frac{a_{n+[nt]}S_{n+[nt]}-a_nS_n}{s_n}\right)_{t\ge0},Y\right)\Longrightarrow((W(1-(1+t)^{3-4p}))_{t\ge0},Y).
\end{align*}
Since $s_n/a_n\sim \sqrt{n/(4p-3)}$, we have
\begin{align*}
\frac{a_{n+[nt]}S_{n+[nt]}-a_nS_n}{s_n}\sim\frac{(a_{n+[nt]}/a_n)S_{n+[nt]}-S_n}{\sqrt{n/(4p-3)}}.
\end{align*}
Hence, by Lemma \ref{Lem1} and $|S_{n+[nt]}|\le n+[nt]$, we obtain
\begin{align*}
&\sup_{t\in[0,T]}
\left|
\frac{(a_{n+[nt]}/a_n)S_{n+[nt]}-S_n}{\sqrt{n}}
-
\frac{(1+t)^{1-2p}S_{n+[nt]}-S_n}{\sqrt{n}}
\right| \\
&\quad=
\sup_{t\in[0,T]}
\frac{|(a_{n+[nt]}/a_n-(1+t)^{1-2p})S_{n+[nt]}|}{\sqrt{n}} \\
&\quad\le
\sup_{t\in[0,T]}
\left|
\frac{a_{n+[nt]}}{a_n}-(1+t)^{1-2p}
\right|
\cdot
\frac{n+[nT]}{\sqrt{n}}=
O(n^{-1/2}) \xrightarrow[n\to\infty]{}0 .
\end{align*}
Therefore, the desired convergence \eqref{mFCLT} follows from the Slutsky's Theorem. The proof is complete.
\end{proof}
\subsection{Diffusive case}
We now prove Theorem \ref{FCLT3stable}. Since it is essentially established in Tokumitsu \cite{Tokumitsu2025functional}, we shall only provide a brief sketch of its proof here.
\begin{proof}[Proof of Theorem \ref{FCLT3stable}]
\begin{align*}
X_{n,k}:=\frac{1}{\Gamma(2p)n^{3/2-2p}}\Delta M_k\quad(n, k\ge1),\quad \mathcal{F}_{n,k}:=\mathcal{F}_{k} \quad(n\ge1, k\ge0).
\end{align*}
Then, $\{X_{n,k},\mathcal{F}_{n,k}\}$ is a martingale difference array, and $\mathcal{F}_{n,k}$ increases as $n$ increases. By setting $k_n(t) = [nt]$ and applying Theorem \ref{martFCLT}, the desired convergence \eqref{sFCLT3} is obtained.
\end{proof}

We now prove Corollary \ref{FCLTjoint1}.
\begin{proof}[Proof of Corollary \ref{FCLTjoint1}]
We apply the Continuous Mapping Theorem (Theorem 5.1 of \cite{Billingsley}). Define a Borel measurable mapping $\Phi:D\times\mathcal{Y}\to D\times D\times\mathcal{Y}$ by
\begin{align*}
\Phi(f,y)
=(f,\phi(f),y)
=\bigl(f,((1+t)^{1-2p}f(1+t)-f(1))_{t\ge0},\,y\bigr).
\end{align*}

Let an arbitrary point $(f,y)\in C\times\mathcal{Y}$, and let
$(f_n,y_n)\in D\times\mathcal{Y}$ be an arbitrary sequence such that
$(f_n,y_n)\to(f,y)$.
Since $f$ is continuous, the convergence $f_n\to f$ in $D$
implies uniform convergence on compact intervals. Therefore, for any $T\ge0$,
\begin{align*}
\sup_{t\in[0,T]}\left|\pi_t\circ\phi(f_n)-\pi_t\circ\phi(f)\right|&=\sup_{t\in[0,T]}\left|(1+t)^{1-2p}f_n(1+t)-(1+t)^{1-2p}f(1+t)\right|\\
&\le\sup_{t\in[0,T]}(1+t)^{1-2p}\sup_{t\in[0,T]}\left|f_n(1+t)-f(1+t)\right|\xrightarrow[n\to\infty]{}0
\end{align*}
where $\pi_t(f)=f(t)$. Hence, $\Phi(f_n,y_n)\xrightarrow[n\to\infty]{}\Phi(f,y)$ holds.

Since Brownian motion has continuous sample paths, we have
\begin{align*}
P\left(\left(\left(\frac{1}{\sqrt{3-4p}}t^{2p-1}W(t^{3-4p})\right)_{t\ge0},Y\right)\in C\times\mathcal{Y}\right)=1.
\end{align*}

Therefore, applying Theorem 5.1 of \cite{Billingsley} to the distributional convergence \eqref{sFCLT3}, we obtain
\begin{align*}
\Phi\left(\left(\frac{S_{[nt]}}{\sqrt{n}}\right)_{t\ge0},Y\right)\Longrightarrow\Phi\left(\left(\frac{1}{\sqrt{3-4p}}t^{2p-1}W(t^{3-4p})\right)_{t\ge0},Y\right).
\end{align*}
The proof is complete.
\end{proof}

\subsection{Critical case}
We now prove Theorem \ref{FCLT4stable}. Since it is essentially established in Tokumitsu \cite{Tokumitsu2025functional}, we shall only provide a brief sketch of its proof here.
\begin{proof}[Proof of Theorem \ref{FCLT4stable}]
\begin{align*}
X_{n,k}:=\frac{1}{\Gamma(3/2)\sqrt{\log n}}\Delta M_k\quad(n, k\ge1),\quad \mathcal{F}_{n,k}:=\mathcal{F}_{k} \quad(n\ge1, k\ge0).
\end{align*}
Then, $\{X_{n,k},\mathcal{F}_{n,k}\}$ is a martingale difference array, and $\mathcal{F}_{n,k}$ increases as $n$ increases. By setting $k_n(t) = [n^t]$ and applying Theorem \ref{martFCLT}, the desired convergence \eqref{sFCLT4} is obtained.
\end{proof}

We now prove Corollary \ref{FCLTjoint2}.
\begin{proof}[Proof of Corollary \ref{FCLTjoint2}]
We apply the Continuous Mapping Theorem (Theorem 5.5 of \cite{Billingsley}). Define a Borel measurable mappings $\Phi_n,\Phi:D\times\mathcal{Y}\to D\times D\times\mathcal{Y}$ by
\begin{align*}
\Phi_n(f,y)
&=(f,\phi_n(f),y)
=\bigl(f,(f(\tau_n(t))-f(1))_{t\ge0},\,y\bigr),\\
\Phi(f,y)&=(f,\phi(f),y)=\bigl(f,(f(1\vee t)-f(1))_{t\ge0},\,y\bigr)
\end{align*}
where $\tau_n(t) := \log_n(n + [n^t])$.

For any $T\ge0$, if $(f_n,y_n)\to (f,y)$ in $D\times\mathcal{Y}$, we have
\begin{align*}
\sup_{t\in[0,T]}
\left|\pi_t\circ\phi_n(f_n)-\pi_t\circ\phi(f)\right|&=\sup_{t\in[0,T]}\left|f_n(\tau_n(t))-f(1\vee t)\right|\\
&\le\sup_{t\in[0,T]}|f_n(\tau_n(t))-f(\tau_n(t))|+\sup_{t\in[0,T]}|f(\tau_n(t))-f(1\vee t)|\\
&\le\sup_{t\in[0,(1\vee T)+1]}|f_n(t)-f(t)|+\sup_{t\in[0,T]}|f(\tau_n(t))-f(1\vee t)|\\
&\xrightarrow[n\to\infty]{}0,
\end{align*}
where we used Lemma \ref{Lem2}. Hence, $\Phi_n(f_n,y_n)\to\Phi(f,y)$ holds.

Since Brownian motion has continuous sample paths, we have
\begin{align*}
P(((W(t))_{t\ge0},Y)\in C\times\mathcal{Y})=1.
\end{align*}

Therefore, applying Theorem 5.5 of \cite{Billingsley} to the distributional convergence \eqref{sFCLT4}, we obtain
\begin{align*}
\Phi_n\left(\left(\frac{S_{[n^t]}}{\sqrt{n^t\log n}}\right)_{t\ge0},Y\right)\Longrightarrow\Phi\left((W(t))_{t\ge0},Y\right).
\end{align*}
The proof is complete.
\end{proof}

\appendix
\section{Appendix}
\subsection{Stable martingale FCLT}
Let $\{X_{n,k}, n\ge1, k\ge1\}$ be an array of random variables defined on a probability space $(\Omega,\mathcal{F},P)$ and let $\{\mathcal{F}_{n,k}, n\ge1, k\ge0\}$ be an array of sub $\sigma$-fields of $\mathcal{F}$ such that for each $n$ and $k\ge1$, $X_{n,k}$ is $\mathcal{F}_{n,k}$ measurable and $\mathcal{F}_{n,i-1}\subset\mathcal{F}_{n,i}$. Suppose $k_n(t)$ is a nondecreasing right continuous function with range $\{0,1,2,\cdots\}$. Set
\begin{align*}
S_{n,0}=0,\quad S_{n,i}=\sum_{k=1}^i X_{n,k}\quad(i\ge1)
\end{align*} 
and
\begin{align*}
Y_n(t)=S_{n,k_n(t)}\quad (t\in[0,\infty)).
\end{align*}
Let $\varepsilon_n$ be a sequence of positive numbers which decrease to zero and define
\begin{align*}
\hat{X}_{n,k}=X_{n,k}1_{\{|X_{n,k}|>\varepsilon_n\}},\quad \bar{X}_{n,k}=X_{n,k}1_{\{|X_{n,k}|\le\varepsilon_n\}},\quad \bar{\bar{X}}_{n,k}=\bar{X}_{n,k}-E[\bar{X}_{n,k}|\mathcal{F}_{n,k-1}].
\end{align*}
Let $\bar{\bar{S}}_{n,i}=\sum_{k=1}^i\bar{\bar{X}}_{n,k}$. We define the time-changed process $W_n:= \bar{\bar{S}}_{n,j_n(\cdot)}$, where the time-change $j_n(t)$ is given by
\begin{align*}
j_n(t)=\sup\left\{j\mid \bar{\bar{V}}_{n,j}:=\sum_{k=1}^jE[\bar{\bar{X}}_{n,k}^2|\mathcal{F}_{n,k-1}]\le t\right\}.
\end{align*}
By Theorem 2.1 of \cite{DurrettResnick}, we have $W_n\Longrightarrow W$ in $D$.

Next set
\begin{align*}
\hat{Y}_n(t)=\sum_{k=1}^{k_n(t)}\hat{X}_{n,k},\quad \bar{Y}_n(t)=\sum_{k=1}^{k_n(t)}\bar{X}_{n,k},\quad A_n(t)=\sum_{k=1}^{k_n(t)}E[\bar{X}_{n,k}|\mathcal{F}_{n,k-1}].
\end{align*}
Let $\varphi_n(t)$ be any strictly increasing continuous function satisfying $\bar{\bar{S}}_{n,k_n(t)}=\bar{\bar{S}}_{n,j_n(\varphi_n(t))}$. (Such a function exists because we can use linear interpolation on the strictly increasing values of the process $(\bar{\bar{V}}_{n,j})$.) Then
\begin{align*}
Y_n(t)=W_n(\varphi_n(t))+\hat{Y}_n(t)+A_n(t).
\end{align*}

The following theorem is a one-dimensional version of the martingale FCLT proved by Touati \cite{TouatimartFCLT}.
\begin{Thm}[Touati {\cite[Corollary of Theorem 2]{TouatimartFCLT}}]\label{martFCLT}
Suppose $\{X_{n,i},\mathcal{F}_{n,i}\}$ is a martingale difference array and the fields $\mathcal{F}_{n,i}$ increases as $n$ increases. If
\begin{itemize}
\item[(a)] for any $t>0$
\begin{align*}
V_n(t):=\sum_{k=1}^{k_n(t)}E[X_{n,k}^2|\mathcal{F}_{n,k-1}]\xrightarrow[n\to\infty]{P}\varphi(t)\quad with\quad P(\text{$\varphi$ is continuous})=1\end{align*}
and
\item[(b)] for any $t>0$ and $\varepsilon>0$,
\begin{align*}
\sum_{k=1}^{k_n(t)}E[X_{n,k}^21_{\{|X_{n,k}|>\varepsilon\}}|\mathcal{F}_{n,k-1}]\xrightarrow[n\to\infty]{P}0.
\end{align*}
\end{itemize} 
Then we have the following joint convergence in $D\times \mathcal{Y}$:
\begin{align}\label{Conv}
((Y_n(t))_{t\ge0},Y)\Longrightarrow(((W\circ\varphi)(t))_{t\ge0},Y)
\end{align}
where $W$ and $(\varphi,Y)$ are independent. 
\end{Thm}
For the completeness of this paper, we give a proof of Theorem \ref{martFCLT} following the framework of Durrett--Resnick \cite{DurrettResnick}.

Let $D_0\subset D$ denote the set of nondecreasing functions, and let $C_0\subset C$ denote the set of strictly increasing functions on $[0,\infty)$.
\begin{proof}
Following the arguments in the proof of Theorem 2.3 in \cite{DurrettResnick}, we have the following convergences as $n \to \infty$:
\begin{align*}
\hat{Y}_n \xrightarrow[n\to\infty]{P} 0, \quad A_n \xrightarrow[n\to\infty]{P} 0, \quad \varphi_n \xrightarrow[n\to\infty]{P} \varphi \quad \text{in } D.
\end{align*}
Let $Y:\Omega\to \mathcal{Y}$ be an arbitrary $\mathcal{F}$-measurable random variable, where $\mathcal{Y}$ is a separable metrizable space. By $W_n\Longrightarrow W$ in $D$ and Theorem 2.4 of  \cite{DurrettResnick}, the sequence $W_n$ is R-mixing (see \cite{DurrettResnick} for the definition). Since $\varphi_n\xrightarrow[n\to\infty]{P}\varphi$ in $D$, we have
\begin{align*}
(\varphi_n,Y)\xrightarrow[n\to\infty]{P}(\varphi,Y)\quad \text{in $D\times \mathcal{Y}$}.
\end{align*} 
Therefore, it follows that
\begin{align}\label{Rmixingconv}
(W_n,\varphi_n,Y)\Longrightarrow (W,\varphi,Y)\quad\text{in $D\times D\times \mathcal{Y}$},
\end{align}
where $W$ and $(\varphi,Y)$ are independent. 

Define
\begin{align*}
\Phi: (C\times D_0)\cup(D\times C_0)\to D,\quad (f,g)\mapsto f\circ g,
\end{align*}
which is continuous with respect to the product Skorokhod topology on $D\times D$ by Theorem 3.1 in \cite{Whitt}. Thus, the following mapping is continuous:
\begin{align*}
\Psi:((C\times D_0)\cup(D\times C_0))\times \mathcal{Y}\to D\times \mathcal{Y}, \quad(f,g,h)\mapsto (f\circ g,h).
\end{align*}
Note that $(W_n,\varphi_n,Y)\in D\times C_0\times \mathcal{Y}$ for each $n$, and by condition (a), we have $(W,\varphi,Y)\in C\times D_0\times \mathcal{Y}$. Therefore, applying the Continuous Mapping Theorem to the distributional convergence \eqref{Rmixingconv}, we obtain 
\begin{align*}
(W_n\circ\varphi_n,Y)\Longrightarrow (W\circ \varphi,Y)\quad\text{in $D\times \mathcal{Y}$}.
\end{align*}
Therefore we have desired convergence \eqref{Conv}. The proof is complete.
\end{proof}
\begin{Rem}
\begin{enumerate}
\item The joint convergence \eqref{Conv} implies that $(Y_n(t))_{t \ge 0}$ converges $\mathcal{F}$-stably to $((W \circ \varphi)(t))_{t \ge 0}$. In the case where $\varphi$ is deterministic, this convergence is actually $\mathcal{F}$-mixing.

\item The R-mixing property of the sequence $W_n$ is equivalent to the mixing convergence of $W_n\Longrightarrow W$; see Corollary 3.3 (v) and Theorem 3.18 (b) of \cite{stableconv} for details.

\item In Touati \cite{TouatimartFCLT}, the stable convergence also includes the process $(V_n)$. However, since this is not needed in this paper, we omit $V_n$ from the statement of Theorem \ref{martFCLT}. By condition (a), we have $V_n \xrightarrow[n\to\infty]{P} \varphi$ in $D$. Therefore, the stable convergence including $V_n$ can also be obtained from Theorem 3.18 (b) of \cite{stableconv}.

\end{enumerate}
\end{Rem}
\subsection{Technical lemmas}
The following lemma gives an estimate for the ratio $a_{n+[nt]}/a_n$. We use it for the calculation after applying Theorem \ref{martFCLT}. However, this estimate should also be useful for other studies of the ERW.
\begin{Lem}\label{Lem1}
Let $p\in[0,1]$. For any $T > 0$, there exists a constant $C_T > 0$ such that for sufficiently large $n$, the following inequality holds:
\begin{align*}
\sup_{t\in[0,T]}\left|\frac{a_{n+[nt]}}{a_n}-(1+t)^{1-2p}\right|\le \frac{C_T}{n}.
\end{align*}
\end{Lem}
\begin{proof}
In the following estimates, the Big-O terms are understood to be uniform with respect to $t \in [0,\infty)$. By Stirling's formula for the gamma function,
\begin{align*}
\Gamma(t)=\sqrt{\frac{2\pi}{t}}\left(\frac{t}{e}\right)^t\left(1+\frac{1}{12t}+\frac{1}{t^2}O(1)\right)\quad\text{($t\to\infty$)}.
\end{align*}
Applying this expansion to $a_n$, we obtain
\begin{align*}
\notag a_n &= \Gamma(2p) \frac{\Gamma(n)}{\Gamma(n+2p-1)} \\
&= \Gamma(2p) n^{1-2p} \left( 1 + \frac{(2p-1)(2-2p)}{2n} + O(n^{-2}) \right)\quad\text{($n\to\infty$)}.
\end{align*}
Then, the ratio $a_{n+[nt]}/a_n$ can be expressed as
\begin{align} \label{ratio}
\frac{a_{n+[nt]}}{a_n} = \left( \frac{n+[nt]}{n} \right)^{1-2p} \cdot\frac{1 + \frac{(2p-1)(2-2p)}{2(n+[nt])} + O(n^{-2})}{1 + \frac{(2p-1)(2-2p)}{2n} + O(n^{-2})}\quad\text{($n\to\infty$)}.
\end{align}
Since $n+[nt] \ge n$, the second factor in \eqref{ratio} is clearly $1+O(n^{-1})$. For the first factor, let $nt = [nt] + d_{n,t}$, where $d_{n,t} \in [0, 1)$ denotes the fractional part of $nt$. Then we have
\begin{align*}
\frac{n+[nt]}{n} = 1 + t - \frac{d_{n,t}}{n}.
\end{align*}
To estimate the term $(1 + t - d_{n,t}/n)^{1-2p}$, we consider the function $f(x)=(1+x)^{1-2p}$ for $x\in (-1,\infty)$. By the Mean Value Theorem, there exists some $\theta_{n,t}\in(t-d_{n,t}/n,t)$ such that
\begin{align*}
f(t)-f(t-d_{n,t}/n)=f'(\theta_{n,t}) \cdot\frac{d_{n,t}}{n}.
\end{align*}
Thus, it follows that
\begin{align*}
\frac{a_{n+[nt]}}{a_n} &=\left((1+t)^{1-2p}-f'(\theta_{n,t})\frac{d_{n,t}}{n}\right)(1+O(n^{-1})) \\
&= (1+t)^{1-2p}+(1+t)^{1-2p}O(n^{-1})-f'(\theta_{n,t})\frac{d_{n,t}}{n}-\left( f'(\theta_{n,t})\frac{d_{n,t}}{n}\right) O(n^{-1})\quad\text{($n\to\infty$)}.
\end{align*}
Note that for $n\ge2$, we have $t-d_{n,t}/n >-1/2$. Since $(1+t)^{1-2p}\le 1\vee(1+T)^{1-2p}$ on $[0,T]$ and the derivative $f'(x)=(1-2p)(1+x)^{-2p}$ is bounded on the interval $[-1/2,\infty)$, we conclude that
\begin{align*}
\sup_{t \in [0,T]} \left| \frac{a_{n+[nt]}}{a_n}-(1+t)^{1-2p}\right| \le \sup_{t\in[-1/2,\infty)} \frac{|f'(t)|}{n}+O(n^{-1})= O(n^{-1})\quad\text{($n\to\infty$)}.
\end{align*}
The proof is complete.
\end{proof}

The following lemma is used in the proof of Corollary \ref{FCLTjoint2}.
\begin{Lem}\label{Lem2}
Let $T \ge0$. $\tau_n(t) = \log_n(n + [n^t])$ converges uniformly to $1 \vee t$ on $[0, T]$.
 Moreover, For sufficiently large $n$, we have $0 \le \sup_{t\in[0,T]}\tau_n(t) \le (1 \vee T) + 1$.
 \end{Lem}
\begin{proof}
First, consider the case $t\in[0,1]$. Since $0\le [n^t]\le n^t$, we have
\begin{align*}
n \le n+[n^t] \le n+n^t.
\end{align*}
Thus, we have
\begin{align*}
1 \le \tau_n(t)
\le \frac{\log(n+n^t)}{\log n}
=1+\frac{\log(1+n^{t-1})}{\log n}\le1+\frac{\log 2}{\log n}.
\end{align*}
Hence, we have
\begin{align*}
\sup_{t\in[0,1]}|\tau_n(t)-1| \xrightarrow[n\to\infty]{} 0.
\end{align*}

Second, consider $t\in[1,T]$. Since $[n^t]\ge n^t-1$, we have
\begin{align*}
n+[n^t] \ge n^t+n-1 \ge n^t(1-n^{-t}).
\end{align*}
Also,
\begin{align*}
n+[n^t] \le n^t+n \le 2n^t.
\end{align*}
Thus, we have
\begin{align*}
t+\frac{\log(1-n^{-t})}{\log n}
\le \tau_n(t)
\le t+\frac{\log 2}{\log n}.
\end{align*}
Therefore, we have
\begin{align*}
\sup_{t\in[1,T]}|\tau_n(t)-t|
\le\max\left\{
\frac{\log 2}{\log n},
\left|\frac{\log(1-n^{-T})}{\log n}\right|
\right\}
\xrightarrow[n\to\infty]{}0.
\end{align*}
The proof is complete.
\end{proof}

\section*{Acknowledgements} 
The author would like to express his sincere gratitude to his supervisor, Professor Kouji Yano, for his dedicated guidance and continuous support throughout the preparation of this paper. Special thanks are due to Professor Masaaki Fukasawa for providing the proof of the stable martingale FCLT. The author would like to express his sincere gratitude to Professor H\'{e}l\`{e}ne Gu\'{e}rin, Ryoichiro Noda and Yuzaburo Nakano for their insightful comments and suggestions.


\end{document}